\documentclass[reqno, 11pt]{amsart}
\usepackage[top=2.5cm,bottom=2.5cm,left=3cm,right=3cm]{geometry}
\input diagxy
\input xy
\usepackage[active]{srcltx} 
\newtheorem{thm}{Theorem}[section]
\newtheorem{cor}[thm]{Corollary}
\newtheorem{lem}[thm]{Lemma}
\newtheorem{prop}[thm]{Proposition}

\theoremstyle{definition}
\newtheorem{defn}[thm]{Definition}
\theoremstyle{remark}
\newtheorem{rem}[thm]{Remark}
\theoremstyle{theorem}

\theoremstyle{example}
\newtheorem{ex}[thm]{Example}
\allowdisplaybreaks

\usepackage[mathscr]{eucal}
\usepackage[all]{xy}
\usepackage{mathrsfs}
\usepackage{xypic}
\usepackage{amsfonts}
\usepackage{amsmath}
\usepackage{amsthm}
\usepackage{amssymb}
\usepackage{latexsym}
\usepackage{tabularx}
\usepackage{graphicx}
\usepackage{pict2e}
\usepackage{hyperref}
\begin{document}
\title[]{Equivariant formality of Hamiltonian transversely symplectic foliations}

\author{Yi Lin}

\address{Yi Lin \\ Department of Mathematical Sciences \\  Georgia Southern University \\Statesboro, GA, 30460 USA}

\email{yilin@georgiasouthern.edu}


\author{Xiangdong Yang}

\address{Xiangdong Yang \\ Department of Mathematics \\ Chongqing University \\ Chongqing 401331 P. R. China}

\email{xiangdongyang2009@gmail.com}


\thanks{The second author is supported by the National Natural Science Foundation of China (No. 11571242 and No. 11701051) and the China Scholarship Council.}


\date{July 18, 2017}



\begin{abstract}
Consider the Hamiltonian action of a compact connected Lie group on a transversely symplectic foliation
which satisfies the transverse hard Lefschetz property.
We establish an equivariant formality theorem and an equivariant symplectic $d\delta$-lemma in this setting.
As an application, we show that if the foliation is also Riemannian,
then there exists a natural formal Frobenius manifold structure on the equivariant basic cohomology of the foliation.
\end{abstract}

\maketitle

\section{Introduction}

Reinhart \cite{R59} introduced the basic cohomology of foliations in late 1950's as a cohomology theory for the space of leaves.
It has become one of fundamental topological invariants for foliations, especially for Riemannian foliations.
An important sub-class of Riemannian foliations are Killing foliations,
as any Riemannian foliation on a simply-connected manifold is Killing.
According to Molino's structure theory \cite{Mo88}, for Killing foliations,
the leaf closures are the orbits of leaves under the action of an abelian Lie algebra of transverse Killing fields,
called the structural Killing algebra.
Goertsches-T\"{o}ben \cite{GT10} introduced the notion of \emph{equivariant basic cohomology},
and used it to study the transverse actions of structural Killing algebras on Killing foliations.
Among other things, they proved a Borel type localization theorem,
and established the equivariant formality  in the presence of a basic Morse-Bott function whose critical set is the union of closed leaves.
As a result, they are able to compute the basic Betti number in many concrete examples,
and relate the basic cohomology to the dynamical  aspects of a foliation.

Let $(M,\eta,g)$ be a compact $K$-contact manifold with a Reeb vector field $\xi$,
and let $T$ be the closure of the Reeb flow in the isometry group $\text{Isom}(M,g)$.
Then $T$ is a compact connected torus.
Moreover, the characteristic Reeb foliation is Killing,
with a structural Killing algebra isomorphic to $\text{Lie}(T)/\text{span}\{\xi\}$.
It is well known that in this situation a generic component of the contact moment map $\Phi: M\rightarrow \mathfrak{t}^*$  is a Morse-Bott function, whose critical set is the union of closed Reeb orbits.
In particular, the results established in \cite{GT10} apply to the transverse actions of the structural Killing algebras on $K$-contact manifolds,
and yield the equivariant formality theorem in this case (cf. \cite{GNT12}).

It is noteworthy that the characteristic foliation of the Reeb vector field of a $K$-contact manifold $(M,\eta,g)$ is \emph{transversely symplectic};
in addition, the transverse action of the structural Killing algebra is Hamiltonian in the sense of Souriau \cite{So97}.
In view of Goertsches-T\"{o}ben's equivariant formality result on $K$-contact manifolds,
one naturally wonders
\emph{if the equivariant formality theorem would continue to hold for a more general class of Hamiltonian actions on transversely symplectic foliations.}

On symplectic manifolds, there are two approaches to proving the Kirwan-Ginzburg equivariant formality theorem.
The first approach is Morse theoretic, which works for arbitrary compact Hamiltonian symplectic manifolds (cf. \cite{Gin87,Kir84}).
The second approach is symplectic Hodge theoretic, which needs to assume that the underlying symplectic manifold has the hard Lefschetz property (cf. \cite{LS04}).
On the upside, it provides an improved version of the equivariant formality theorem, which asserts that any de Rham cohomology class has a canonical equivariant extension.

In an accompanying paper \cite{L16}, the first author extended symplectic Hodge theory to any transversely symplectic manifold with the \emph{transverse $s$-Lefschetz property}, and established the symplectic $d\delta$-lemma in this framework.
In the present article, for Hamiltonian actions of compact connected Lie groups on transversely symplectic foliations,
we apply the symplectic Hodge theory to prove the following result.
\begin{thm}[Theorem \ref{eq-formality}]\label{main thm1}
Consider the Hamiltonian action of a compact connected Lie group $G$ on a compact transversely symplectic foliation $(M, \mathcal{F}, \omega)$.
Suppose that $(M,\mathcal{F},\omega)$ satisfies the transverse hard Lefschetz property.
Then there is a canonical $S(\mathfrak{g}^{*})^{G}$-module isomorphism from the equivariant basic cohomology $H_{G}(M,\mathcal{F})$ to
$
S(\mathfrak{g}^{*})^{G}\otimes H(M,\mathcal{F}).
$
\end{thm}
It is important to note that on a transversely symplectic foliation, components of a moment map are in general not Morse-Bott functions,
unless the action satisfies the so called \emph{clean condition} discovered by Lin-Sjamaar in \cite{LS16}.
However, a striking feature of our Hodge theoretic approach is that it would continue to work, even when the action is not clean,
as long as the transverse hard Lefschetz property is satisfied.

On a compact symplectic manifold with the hard Lefschetz property, Merkulov \cite{Mer98} established the symplectic $d\delta$-lemma,
and used it to produce a formal Frobenius manifold structure on the de Rham cohomology of the symplectic manifold.
Independently, Cao-Zhou \cite{CZ99,CZ00} proved similar results on the ordinary and equivariant de Rham cohomology of K\"ahler manifolds.
For Hamiltonian Lie group actions on transversely symplectic foliations with the transvere hard Lefschetz property,
our method yields an equivariant version of the symplectic $d\delta$-lemma on basic forms.
As an application of this result,
we show that there is a formal Frobenius manifold structure on the equivariant basic cohomology of the foliation (Theorem \ref{construction-Frobenius-manifolds-eq-case}).
This simultaneously generalizes the constructions of Merkulov and Cao-Zhou.

Transversely symplectic foliations are naturally related to different areas in differential geometry.
Reeb characteristic foliations in both contact and co-symplectic geometries are clearly transversely symplectic.
Moreover, leaf spaces of transversely symplectic foliations include symplectic orbifolds (in the sense of Satake \cite{Sa57}) and symplectic quasi-folds \cite{Pra01} as special examples.
In many known cases, transversely symplectic foliations arise as taut K\"ahler foliations,
which are known  to have the transverse hard Lefschetz property (cf. \cite{Ka90}).
The results proved in this paper apply to these situations,
and yield new examples of $dGBV$-algebras whose cohomologies carry the structure of a formal Frobenius manifold.

This paper is organized as follows.
In Section 2 we review symplectic Hodge theory on transversely symplectic foliations.
In Section 3, we establish an equivariant formality theorem for the Hamiltonian action of  a compact connected Lie group on a transversely symplectic foliation.
We also obtain an equivariant version of the symplectic $d\delta$-lemma on transversely symplectic foliations.
In Section 4, we show that there exists a formal Frobenius manifold structure on the equivariant basic cohomology of a Hamiltonian transversely symplectic foliation that satisfies the transverse hard Lefschetz property.
In Section 5, we present some concrete examples of transversely symplectic foliations, which are also Riemannian, and which satisfy the transverse hard Lefschetz property.

\subsection*{Acknowledgements}
Much of this joint work is completed while the first author was visiting Sichuan University in the spring of 2016.
He would like to  thank the School of Mathematics and the geometry and topology group there for providing him with an excellent working environment.
The second author would like to thank Prof. Guosong Zhao for his constant encouragement and moral support over the years.
Both authors are grateful to  Prof. Xiaojun Chen for his interest in this work, and for many useful discussions.
Finally, the authors would like to thank the anonymous referee for helpful comments and suggestions.

\section{Hodge theory on transversely symplectic foliations}

In this section, we review the elements of transversely symplectic Hodge theory to set up the stage.
We refer to \cite{Br88} and \cite{Y96} for general background on symplectic Hodge theory,
and to \cite{L16} for a detailed exposition on symplectic Hodge theory on foliations.

Assume that $\mathcal{F}$ is a foliation on a smooth manifold $M$ of co-dimension $m$.
Let $\Xi(M)$ be the Lie algebra of smooth vector fields on $M$,
and let $\Xi(\mathcal{F})\subset\Xi(M)$ be the Lie sub-algebra of vector fields which are tangent to the leaves of $\mathcal{F}$.
We say that an element $X\in\Xi(M)$ is \emph{foliate}, if $[X,Y]\in\Xi(\mathcal{F})$ for any $Y\in\Xi(\mathcal{F})$.
In particular, the set of foliate fields, denoted by $L(M,\mathcal{F})$, is a Lie sub-algebra of $\Xi(M)$,
since it is the normalizer of $\Xi(\mathcal{F})$ in $\Xi(M)$.
A \emph{transverse vector field} is a smooth section of   $TM/T\mathcal{F}$ that is induced by a foliate vector field.
It is easy to see that the set of transverse fields $l(M,\mathcal{F})=L(M,\mathcal{F})/\Xi(\mathcal{F})$ also admits a Lie algebra structure with
an induced Lie bracket from $L(M,\mathcal{F})$.

The space of \emph{basic forms} on $M$ is defined as follows.
$$
\Omega(M,\mathcal{F})=\bigl\{\alpha\in\Omega(M)\mid\iota(X)\alpha=\mathcal{L}(X)\alpha=0,\,\text{for all}\,X\in\Xi(\mathcal{F})\bigr\}.
$$
Since the exterior differential operator $d$ preserves basic forms,
we obtain a sub-complex of the de Rham complex $\{\Omega^{*}(M),d\}$,
called the \emph{basic de Rham complex} as follows.
$$
\xymatrix@C=0.5cm{
  \cdots \ar[r] & \Omega^{k-1}(M,\mathcal{F}) \ar[r]^{\quad d} & \Omega^{k}(M,\mathcal{F}) \ar[r]^{d} & \Omega^{k+1}(M,\mathcal{F}) \ar[r]^{\qquad d} & \cdots }.
$$

The cohomology of the basic de Rham complex $\{\Omega^{*}(M,\mathcal{F}),d\}$,
denoted by $H(M,\mathcal{F})$, is called the \emph{basic cohomology} of $M$ with respect to the foliation $\mathcal{F}$.
If $M$ is connected, then $H^{0}(M,\mathcal{F})\cong\mathbb{R}^{1}$.
In general, the group $H^{k}(M,\mathcal{F})$ may be infinite-dimensional for $k\geq2$.
However, if $M$ is a closed oriented manifold and if $\mathcal{F}$ is a Riemannian foliation,
then the basic cohomology is finite-dimensional;
moreover, we have either $H^{m}(M,\mathcal{F})=0$ or $H^{m}(M,\mathcal{F})=\mathbb{R}$ (cf. \cite[Th\'{e}or\`{e}m 0.]{KSH85}).
In particular, a Riemannian foliation $\mathcal{F}$ on a closed manifold $M$ is said to be \emph{taut}, if $H^{m}(M,\mathcal{F})=\mathbb{R}$.

\begin{defn}\label{transverse-sym}(\cite{H70})
Let $\mathcal{F}$ be a foliation on a smooth manifold $M$,
and let $P$ be the integrable subbundle of $TM$ associated to $\mathcal{F}$.
We say that $\mathcal{F}$ is a \emph{transversely symplectic foliation},
if there exists a closed 2-form $\omega$, called the \emph{transversely symplectic form},
such that for each $x\in M$, the kernel of $\omega_x$ coincides with the fiber of $P$ at $x$.
\end{defn}

Let $(M,\mathcal{F},\omega)$ be a transversely symplectic foliation of co-dimension $2n$.
The transversely symplectic form $\omega$ induces a non-degenerate bi-linear paring $B(\cdot,\cdot)$ on $\Omega^{p}(M,\mathcal{F})$,
which in turn gives rise to the \emph{symplectic Hodge star operator} $\star$ on $\Omega^{p}(M,\mathcal{F})$ as follows.
$$
\beta\wedge\star \alpha=B(\alpha,\beta)\frac{\omega^{n}}{n!},
$$
for any $\alpha,\beta\in\Omega^{p}(M,\mathcal{F})$.
The bi-linear pairing $B(\cdot, \cdot)$ is symmetric when $p$ is even, and skew-symmetric when $p$ is odd.
It follows easily from the definition that
\begin{equation} \label{star-adjoint-identity}
\beta \wedge \star \alpha = \star \beta \wedge \alpha,\,\,\,\,\star^2=\textmd{id}.
\end{equation}
The transpose operator $\delta$ of $d$ is defined by
$$
  \delta:\Omega^{p}(M,\mathcal{F}) \rightarrow \Omega^{p-1}(M,\mathcal{F}),\,\,\,
  \alpha \mapsto (-1)^{p+1}\star d\star\alpha.
$$

By definition, it is easy to see that the operator $\delta$ satisfies the equations $\delta^{2}=0$ and $d\delta+\delta d=0$.
In this context,  a basic form $\alpha$ is called  (symplectic) \emph{harmonic} if it satisfies $d\alpha=\delta\alpha=0$.
Set
$$
\Omega_{\textmd{har}}(M,\mathcal{F})=\bigl\{\alpha\in\Omega(M,\mathcal{F})\,\mid\, d\alpha=\delta\alpha=0\bigr\}.
$$

There are three important operators acting on the space of basic forms:
\begin{enumerate}
  \item [(1)]$L:\Omega^{*}(M,\mathcal{F})\rightarrow\Omega^{*+2}(M,\mathcal{F}),\,\,\,\alpha\mapsto\alpha\wedge\omega$,
  \item [(2)]$\Lambda:\Omega^{*}(M,\mathcal{F})\rightarrow\Omega^{*-2}(M,\mathcal{F}),\,\,\,\alpha\mapsto\star L\star\alpha$,
  \item [(3)]$H:\Omega^{k}(M,\mathcal{F})\rightarrow\Omega^{k}(M,\mathcal{F}),\,\,\,\alpha\mapsto(n-k)\alpha$.
\end{enumerate}
In particular, we have the following result.
\begin{lem}\label{Leibniz-rule}
Let $f$ be a basic function, and $X$ a foliate vector field such that $\iota(X) \omega=df$.
Then for any basic form $\alpha$ we have
\begin{itemize}
\item [a)] $[\Lambda, \iota(X)]\alpha=0$.
\item [b)] $\delta(f\alpha)=f\delta\alpha-\iota(X)\alpha.$
\item [c)] $\delta( df\wedge \alpha)=-df\wedge \delta\alpha+\mathcal{L}(X)\alpha$.
\end{itemize}
\end{lem}

\begin{proof}
The assertion a) is a direct consequence of \cite[Lemma 3.2]{L16},
and b) can be proved by the same argument as the one used in \cite[Proposition 2.5]{LS04}.
It remains to check the assertion c).
Using b) and the identity $d\delta+\delta d=0$, we have
\[
\begin{split}
\delta(df\wedge \alpha)
&= \delta\left( d(f\alpha)-f d\alpha\right)\\
&=-d\delta (f\alpha)-\delta (fd\alpha)\\
&=-d\left(f\delta\alpha-\iota(X)\alpha\right)-f\delta d\alpha+\iota(X)d\alpha\\
&=-d(f\delta\alpha)-f\delta d\alpha+\left(d\iota(X)+\iota(X) d\right)\alpha\\
&= -df\wedge\delta \alpha -f\left(d\delta+\delta d\right)\alpha+\mathcal{L}(X)\alpha\\
&= -df\wedge \delta \alpha+\mathcal{L}(X)\alpha.
\end{split}
\]
This proves the assertion c).
\end{proof}

A straightforward calculation yields the following commutator relations.
\begin{prop}[cf. {\cite[Lemma 3.2]{L16}}]
\[\begin{split}&
[L,d]=0,\,\,\,[\Lambda,d]=\delta,\,\,[\Lambda,\delta]=0, \,\,[L,\delta]=-d ;\\&
[L,\Lambda]=H,\,\,[H,L]=-2L,\,\,[H,\Lambda]=2\Lambda. \qquad
\end{split}\]
\end{prop}

\begin{defn}
Let $(M,\mathcal{F},\omega)$ be a transversely symplectic foliation of co-dimension $2n$.
We say that $M$ satisfies the \emph{transverse hard Lefschetz property}, if for any $0\leq k\leq n$, the map
$$
L^{k}:H^{n-k}(M,\mathcal{F})\rightarrow H^{n+k}(M,\mathcal{F})
$$
is an isomorphism.
\end{defn}

On compact symplectic manifolds, Brylinski \cite{Br88} conjectured that every de Rham cohomology class has a symplectic harmonic representative.
However, Mathieu \cite{Ma95} proved that this conjecture is true if and only if the manifold satisfies the hard Lefschetz property.
Mathieu's theorem was sharpened by Merkulov \cite{Mer98} and Guillemin \cite{Gui01}, who independently established the symplectic $d\delta$-lemma. The symplectic $d\delta$-lemma was first extended to transversely symplectic flows by Zhenqi He \cite{H10},
and more recently, by the first author \cite{L16} to arbitrary transversely symplectic foliations.
The following results are reformulations of \cite[Theorem 4.1, 4.8]{L16}.
\begin{thm}\label{thm2.15}
Let $(M,\mathcal{F},\omega)$ be a transversely symplectic foliation with the transverse hard Lefschetz property.
Then every basic cohomology class has a symplectic harmonic representative.
\end{thm}

\begin{thm}\label{symplectic-d-delta}
Assume that $(M,\mathcal{F},\omega)$ is a transversely symplectic foliation that satisfies the transverse hard Lefschetz property.
Then on the space of basic forms we have
$$
\emph{im}\,d\cap\ker\,\delta=\ker\,d\cap \emph{im}\,d=\emph{im}\,d\delta.
$$
\end{thm}

Let $\Omega_{\delta}(M,\mathcal{F})=\text{ker}\,\delta \cap \Omega(M,\mathcal{F})$.
Since $d$ anti-commutes with $\delta$, the subspace $\Omega_{\delta}(M,\mathcal{F})$
forms a sub-complex of the basic de Rham complex $\{\Omega(M,\mathcal{F}),d\}$, the cohomology of which we denote by $H_{\delta}(M,\mathcal{F})$.
The following result is a direct consequence of Theorem \ref{symplectic-d-delta}.
Here $H(\Omega(M,\mathcal{F}),\delta)$ denotes the homology of $\Omega(M,\mathcal{F})$ with respect to $\delta$.

\begin{thm} \label{formality-basic-complex}
Assume that $(M,\mathcal{F},\omega)$ is a transversely symplectic foliation that satisfies the transverse hard Lefschetz property.
Then the $d$-chain maps in the diagram
\[
\Omega(M,\mathcal{F}) \longleftarrow \Omega_{\delta}(M,\mathcal{F})\longrightarrow H(\Omega(M,\mathcal{F}), \delta)
\]
are quasi-isomorphisms that induce isomorphisms in cohomology.
\end{thm}

\section{Equivariant formality and basic $d_{G}\delta$-lemma}\label{eq-formality-ddelta-lemma}

In this section we study the equivariant basic cohomology of Hamiltonian actions on transversely symplectic foliations using the Hodge theoretic approach.
Let $\mathfrak{g}$ be a finite-dimensional Lie algebra. Recall that a \emph{transverse action} of $\mathfrak{g}$ on a foliated manifold $(M,\mathcal{F})$ is defined to be a Lie algebra homomorphism $\mathfrak{g}\rightarrow l(M,\mathcal{F})$ (cf. \cite[Definition 2.1]{GT10}).
We propose the following definition of transverse actions of a Lie group $G$.
\begin{defn}\label{trans-G-action}
Consider the action of a Lie group $G$ with the Lie algebra $\mathfrak{g}$ on a foliated manifold $(M,\mathcal{F})$.
We say that the action of $G$ is \emph{transverse}, if the image of the associated infinitesimal action map $\mathfrak{g}\rightarrow \Xi(M)$ lies in $L(M,\mathcal{F})$.
\end{defn}

\begin{rem}
Suppose that there is a transverse action of a Lie group $G$  with Lie algebra $\mathfrak{g}$ on a foliated manifold $(M,\mathcal{F})$.
Then by definition we have the following commutative diagram of Lie algebra homomorphisms.
$$
\xymatrix{
  \mathfrak{g} \ar[dr]_{} \ar[r]^{}
                & L(M,\mathcal{F}) \ar[d]^{\textmd{pr}}  \\
                & l(M,\mathcal{F})             }
$$
Here the vertical map is the natural projection.
Therefore we also have a transverse $\mathfrak{g}$-action on $(M,\mathcal{F})$ in the sense of \cite[Definition 2.1]{GT10}.
\end{rem}

\begin{lem}\label{preserve-basic}
Consider the transverse action of a compact connected Lie group $G$ on a foliated manifold $(M,\mathcal{F})$.
If $\alpha$ is a basic form, and if $X_M$ is a fundamental vector field induced by an element $X\in \mathfrak{g}$,
then $\iota(X_M)\alpha$ and $\mathcal{L}(X_M)\alpha$ are also basic forms.
\end{lem}
\begin{proof}
Let $Y\in \Xi(\mathcal{F})$.
Since the action of $G$ is transverse, we get $[Y,X_M]\in \Xi(\mathcal{F})$.
It follows that
$$
\iota(Y)\iota(X_M)\alpha=-\iota(X_M)\iota(Y)\alpha=0,
$$
and that
$$
\mathcal{L}(Y)\iota(X_M)\alpha=\iota([Y,X_M])\alpha+\iota(X_M)\mathcal{L}(Y)\alpha=0.
$$
This proves that $\iota(X_M)\alpha$ is a basic form.
A similar calculation shows that $\mathcal{L}(X_M)\alpha$ is also basic.
\end{proof}

Suppose that there is a transverse action of a compact connected Lie group $G$ on a foliated manifold $(M,\mathcal{F})$.
As an immediate consequence of Lemma \ref{preserve-basic}, we see that $\Omega(M,\mathcal{F})$ is a $G^{\star}$-module in the sense of \cite[Definition 2.3.1]{GS99}.
Therefore, there is a well defined Cartan model of $\Omega(M,\mathcal{F})$  given by
\[
\Omega_{G}(M,\mathcal{F}):=[S(\mathfrak{g^{*}})\otimes\Omega(M,\mathcal{F})]^{G},
\]
which we call the \emph{equivariant basic Cartan complex}.

To simplify the notations, let us write  $\Omega_{\textmd{bas}}=\Omega(M,\mathcal{F})$,
and $\Omega_{G,\textmd{bas}}=\Omega_{G}(M,\mathcal{F})$.
Elements of $\Omega_{G,\textmd{bas}}$ can be regarded as equivariant polynomial maps from $\mathfrak{g}$ to $\Omega_{\textmd{bas}}$,
and are called \emph{equivairant basic differential forms} on $M$.
The equivariant basic Cartan model $\Omega_{G,\textmd{bas}}$ has a bi-grading given by
$$
\Omega^{i,j}_{G,\textmd{bas}}=[S^{i}(\mathfrak{g^{*}})\otimes\Omega^{j-i}_{\textmd{bas}}]^{G};
$$ moreover, it is quipped with the vertical differential $1\otimes d$, which we abbreviate to $d$,
and the horizontal differential $\partial$, which is defined by
\[\partial(\alpha(\xi))=-\iota(\xi)\alpha(\xi),\,\,\,\textmd{for all}\,\xi\in\frak{g}.\]
Here $\iota(\xi)$ denotes the inner product with the fundamental vector field on $M$ induced by $\xi\in \mathfrak{g}$.
As a single complex, $\Omega_{G,\textmd{bas}}$ has a grading given by
$$
\Omega_{G,\textmd{bas}}^k=\displaystyle \bigoplus_{i+j=k}\Omega_{G,\textmd{bas}}^{i,j},
$$
and a total differential
$d_G=d+\partial$, which is called the equivariant exterior differential.
We say that an equivariant differential basic form $\alpha$ is equivariantly closed,
resp., equivariantly exact, if $d_G\alpha=0$, resp. $\alpha=d_G\beta$ for some equivariant basic form $\beta$.

\begin{defn}
The \emph{equivariant basic cohomology} of the transverse $G$-action on $(M,\mathcal{F})$ is defined to be the total cohomology of the equivariant basic Cartan complex $\{\Omega_{G}(M,\mathcal{F}),d_{G}\}$, which is denoted
by $H_{G}(M,\mathcal{F})$.
\end{defn}

We would like to point out that the above definition of equivariant basic cohomology was first introduced by Goertsches-T\"{o}ben in \cite{GT10} using the language of equivariant cohomology of $\mathfrak{g}^{\star}$-algebras.
Following Goresky-Kottwitz-MacPherson \cite{GKM98}, we propose the following definition of equivariant formality for transverse $G$-actions.
\begin{defn}
A transverse $G$-action on $(M,\mathcal{F})$ is \emph{equivariantly formal} if
$$
H_{G}(M,\mathcal{F})\cong S(\mathfrak{g}^*)^G\otimes H(M,\mathcal{F})
$$
as graded $S(\mathfrak{g}^{*})^{G}$-modules.
\end{defn}

Next, we review the notion of Hamiltonian $G$-actions on transversely symplectic foliations.

\begin{defn}(\cite{LS16})\label{Hamiltonian}
Consider the action of a compact connected Lie group $G$ with the Lie algebra $\mathfrak{g}$ on a transversely symplectic foliation $(M,\mathcal{F},\omega)$.
We say that the $G$-action on $(M,\mathcal{F},\omega)$ is \emph{Hamiltonian},
if the $G$-action preserves the transversely symplectic form $\omega$,
and if there exists an equivariant map,
$$
 \Phi:M\rightarrow\mathfrak{g}^{*},
$$
called a moment map, such that $d\langle\Phi,\xi\rangle=\iota(\xi)\omega$, for each $\xi\in\mathfrak{g}$.
Here $\langle\cdot,\cdot\rangle$ denotes the dual pairing between $\mathfrak{g}$ and $\mathfrak{g}^{*}$.
\end{defn}
\begin{rem}
By definition, the Hamiltonian action of a Lie group $G$ on a transversely symplectic manifold $(M,\mathcal{F},\omega)$ is always transverse. Indeed, since the action preserves the transversely symplectic form $\omega$, it also preserves its null foliation $\mathcal{F}$.
It then follows from \cite[Proposition 2.2]{Mo88} that the $G$-action must be transverse.
\end{rem}
From now on, we assume that $(M,\mathcal{F},\omega)$ is a compact transversely symplectic foliation that satisfies the transverse hard Lefschetz property,
and that there is a compact connected Lie group $G$ acting on $(M,\mathcal{F},\omega)$ in a Hamiltonian fashion
with a moment map $\Phi: M\rightarrow \mathfrak{g}^*$, where $\mathfrak{g}=\textmd{Lie}(G)$.
The symplectic Hodge theory gives rise to a third differential $1\otimes \delta$ on $\Omega_{G,\textmd{bas}}$,
which we will abbreviate to $\delta$.

\begin{lem}\label{delta-anti-commutes-dG}
On the space of equivariant basic differential forms $\Omega_{G,\emph{bas}}$, the following identities hold.
\[\partial \delta=-\delta\partial,\,\,\,d_G\delta=-\delta d_G.\]
\end{lem}
\begin{proof}
It was shown in \cite[Lemma 3.1]{LS04} that $\partial \delta=-\delta\partial$ and $d_G\delta=-\delta d_G$ hold on the space of equivariant differential forms.
Since $d_G$, $\delta$ and $\partial$ map basic forms to basic forms,
these two identities also hold on the space of equivariant basic differential forms.
\end{proof}

This implies that $\Omega^{\delta}_{G,\textmd{bas}}:=\text{ker}\,\delta\cap \Omega_{G,\textmd{bas}}$ is a \emph{double sub-complex} of $\Omega_{G,\textmd{bas}}$, and that the homology $H(\Omega_{G,\textmd{bas}},\delta)$ with respect to $\delta$ is a double complex with the differentials induced by $d$ and $\partial$.
Thus we have a diagram of morphisms of double complexes
\begin{equation}\label{diagram-morphisms}
\Omega_{G,\textmd{bas}}\longleftarrow \Omega^{\delta}_{G,\textmd{bas}}\longrightarrow H(\Omega_{G,\textmd{bas}},\delta).
\end{equation}
Since $\delta$ acts trivially on the polynomial part, these morphisms in (\ref{diagram-morphisms}) are actually morphisms of $S(\mathfrak{g}^*)^G$-modules.

We first establish a preliminary result about the action of $\iota(\xi)$ on invariant basic forms.
Let $\Omega^{G}_{\textmd{bas}}$ be the space of $G$-invariant basic forms on $M$.
The Cartan's identity $\mathcal{L}(\xi)=\iota(\xi)d+d\iota(\xi)$ implies that the morphism
$\iota(\xi): \Omega_{\textmd{bas}}^G\rightarrow \Omega_{\textmd{bas}}^G$ is a chain map with respect to $d$. Here $\mathcal{L}(\xi)$ denotes the Lie derivative of the fundamental vector field on $M$ induced by $\xi\in \mathfrak{g}$.
Similarly, an application of the identity $\delta\partial+\partial\delta=0$ to the zeroth column of $\Omega_{G,\textmd{bas}}$ implies that $\iota(\xi)$ is a chain map with respect to $\delta$.

\begin{lem}\label{lem3.4}
Let $\xi \in\mathfrak{g}$ and $\alpha\in \Omega_{\emph{bas}}^G$.
If $\alpha$ is $d$-closed, then $\iota(\xi)\alpha$ is $d$-exact.
If $\alpha$ is $\delta$-closed, then $\iota(\xi)\alpha$ is $\delta$-exact.
\end{lem}
\begin{proof}
Since the action of $G$ is Hamiltonian, it follows from \cite[Proposition 2.5]{LS04} that
\begin{equation}\label{lebeniz-rule}
\iota(\xi)\alpha=\Phi^{\xi}(\delta\alpha)-\delta(\Phi^{\xi}\alpha)
\end{equation}
where $\Phi^{\xi}$ is the $\xi$-component of the moment map $\Phi:M\rightarrow \mathfrak{g}^*$.
If $\alpha$ is $\delta$-closed, then
we have that $\iota(\xi)\alpha=-\delta (\Phi^{\xi}\alpha)$.
Since $\Phi^{\xi}$ is a basic function, we get that $\iota(\xi)\alpha$ is $\delta$-exact in $\Omega_{\textmd{bas}}^G$.

It remains to show that if $\alpha\in \Omega_{\textmd{bas}}^G$ is a $d$-closed basic $k$-form, then $\iota(\xi)\alpha$ is $d$-exact.
Since $M$ satisfies the transverse hard Lefschetz property, by \cite[Theorem 4.3]{L16},
for each class $[\alpha]\in H^k(M,\mathcal{F})$ there exists a unique primitive decomposition
$$
[\alpha]=\displaystyle \sum_{r}L^{r} [\alpha_r].
$$
Here $[\alpha_r]\in H^{k-2r}(M,\mathcal{F})$ is a primitive basic cohomology class, i.e., $L^{n-k+2r+1}[\alpha]=0$.
However, since the action is Hamiltonian, we have
$$
\iota(\xi)(\omega\wedge \alpha)=d\Phi^{\xi}\wedge \alpha+\omega\wedge \iota(\xi)\alpha.
$$
Thus to finish the proof, it suffices to show that $\iota(\xi)\alpha$ is exact when $[\alpha]$ is a primitive basic cohomology class.
We note that the argument given in \cite[Lemma 3.2]{LS04} continues to hold in the present situation to show the exactness of $\iota(\xi)\alpha$.
\end{proof}

Note that the symplectic $d\delta$-lemma, Theorem \ref{symplectic-d-delta}, holds for equivariant basic differential forms as well as for ordinary basic differential forms.
In particular, the inclusion $\Omega_{\textmd{bas}}^G\hookrightarrow \Omega_{\textmd{bas}}$ is a deformation retraction for $\delta$ as well as for $d$.
The same argument as given in the proof of \cite[Lemma 3.3.]{LS04} provides us the following result.
\begin{lem}\label{trivial-double-complex}
The differentials induced by $d$ and $\partial$ on $H(\Omega_{G,\emph{bas}},\delta)$ are 0.
Moreover, we have the isomorphism
\begin{equation}\label{homology-complex}
H(\Omega_{G,\emph{bas}},\delta)\cong S(\mathfrak{g}^*)^G\otimes H(M,\mathcal{F}).
\end{equation}
\end{lem}

We are now in a position to prove the equivariant formality property of Hamiltonian actions on transversely symplectic foliations.
\begin{thm}\label{eq-formality}
Let $(M,\mathcal{F},\omega)$ be a compact transversely symplectic manifold that
satisfies the transverse hard Lefschetz property,
and let a compact connected Lie group $G$ act on $M$ in a Hamiltonian fashion.
Then the morphisms (\ref{diagram-morphisms}) induces isomorphisms of $S(\mathfrak{g}^*)^G$-modules
$$
H_{G}(M,\mathcal{F})\xleftarrow{\,\,\cong\,\,} H(\Omega^{\delta}_{G,\emph{bas}},d_{G})\xrightarrow{\,\,\cong\,\,} H(\Omega_{G,\emph{bas}},\delta).
$$
\end{thm}

\begin{proof}
We first note that since $G$ is connected, the identity $\mathcal{L}(\xi)=d\iota(\xi)+\iota(\xi)d$ together with the identity (\ref{lebeniz-rule}) implies that $G$ acts trivially on both $H(M,\mathcal{F})$ and $H(\Omega(M,\mathcal{F}),\delta)$.
Let $E$ be the spectral sequence of $\Omega_{G,\textmd{bas}}$ relative to the filtration associated to the horizontal grading and $E_{\delta}$ that of $\Omega^{\delta}_{G,\textmd{bas}}$.
The first terms are
\begin{eqnarray}
E_1&=&\text{ker}d/\text{im}\, d
    = \left[S(\mathfrak{g}^*)\otimes H(M,\mathcal{F})\right]^G
    =S(\mathfrak{g}^*)^G\otimes H(M,\mathcal{F}) \label{first-terms-a}\\
(E_{\delta})_1
   &=&(\text{ker}\, d\cap\text{ker}\, \delta)/(\text{im}\, d\cap \text{ker}\,\delta) \nonumber \\
   &=&\left[S(\mathfrak{g}^*)\otimes H(\Omega(M,\mathcal{F}),\delta)\right]^G=S(\mathfrak{g}^*)^G\otimes H(M,\mathcal{F}). \label{first-terms-b}
\end{eqnarray}

Here we used the observation we made in the paragraph right before Lemma \ref{trivial-double-complex},
as well as the isomorphism $H(\Omega(M,\mathcal{F}),\delta)\cong H(M,\mathcal{F})$ of Theorem \ref{formality-basic-complex}.
By Lemma \ref{trivial-double-complex}, $H(\Omega_{G,\textmd{bas}},\delta)$ is a trivial double complex,
its spectral sequence is therefore constant with trivial differentials at each stage.
The two morphisms in (\ref{diagram-morphisms}) induce morphisms of spectral sequences
\[
E\longleftarrow E_{\delta}\longrightarrow H(\Omega_{G,\textmd{bas}},\delta).
\]
It follows from (\ref{homology-complex}), (\ref{first-terms-a}) and \eqref{first-terms-b} that these morphisms induce isomorphisms at the first stage.
Thus they must induce isomorphisms at every stage.
In particular, these three spectral sequences converge to the same limit,
and so the morphisms (\ref{diagram-morphisms}) induce isomorphisms on total cohomology.
This completes the proof.
\end{proof}

An argument similar to the one used in \cite[Theorem 3.9]{LS04} gives us the following equivariant version of the symplectic $d\delta$-lemma on transversely symplectic manifolds.
\begin{thm}\label{dG-delta-lemma}
Let $\alpha\in\Omega_{G,\emph{bas}}$ be an equivariant basic form satisfying $d_{G}\alpha=0$ and $\delta\alpha=0$.
If $\alpha$ is either $d_{G}$-exact or $\delta$-exact, then there exists $\beta\in\Omega_{G,\emph{bas}}$ such that $\alpha=d_{G}\delta\beta$.
\end{thm}

We now discuss the implications of Theorem \ref{eq-formality}.
Observe that $\Omega_{G,\textmd{bas}}^{0,k}=(\Omega_{\textmd{bas}}^k )^G$,
the space of $G$-invariant basic $k$-forms on $M$.
Thus the zeroth column of the basic Cartan model is the $G$-invariant basic de Rham complex $\Omega_{\textmd{bas}}^G$,
which is a deformation retraction of the basic de Rham complex because $G$ is connected.
Therefore, we have an isomorphism $H(\Omega_{\textmd{bas}}^G)\cong H(M,\mathcal{F})$.
The natural projection map $\overline{p}: \Omega_{G,\textmd{bas}}\rightarrow \Omega^G_{\textmd{bas}}$,
defined by $\overline{p}(\alpha)=\alpha(0)$, is a chain map with respect to the equivariant exterior derivative $d_G$ on $\Omega_{G,\textmd{bas}}$ and the ordinary exterior derivative $d$ on $\Omega_{\textmd{bas}}$.
It induces a morphism of cohomology groups
$
p:H_{G}(M,\mathcal{F})\rightarrow H(M,\mathcal{F}).
$
Theorem \ref{eq-formality} implies that the spectral sequence $E$ degenerates at the first stage, and that the map $p$ is surjective.
In other words, every basic cohomology class can be extended to an equivariant basic cohomology class.
However, Theorem \ref{eq-formality} would also imply that there is a canonical choice of such an extension.
Let
\begin{equation}\label{canonical-section}
s: H(M,\mathcal{F})\rightarrow H_{G}(M,\mathcal{F})
\end{equation}
be the composition of the map
$$
H(M,\mathcal{F})\rightarrow S(\mathfrak{g}^*)^G\otimes H(M,\mathcal{F})
$$
which sends a cohomology class $a$ to $1\otimes a$, and the isomorphism
$$
S(\mathfrak{g}^*)^G\otimes H(M,\mathcal{F}) \rightarrow H_{G}(M,\mathcal{F})
$$
given by Theorem \ref{eq-formality}.
The following result is a direct consequence of Theorem \ref{formality-basic-complex} and Theorem \ref{eq-formality}.

\begin{cor}
The map $s$ is a section of $p$.
Thus every basic cohomology class can be extended to a equivariant basic cohomology class in a canonical way.
\end{cor}
\begin{proof}
For details of the proof see \cite[Corollary 3.5]{LS04}.
\end{proof}
\section{Formal Frobenius manifolds modelled on equivariant basic cohomology}
Consider the Hamiltonian action of a compact connected Lie group on a transversely symplectic foliation.
In this section, following the approach initiated by Barannikov-Kontsevich \cite{BK98},
we show that if the foliation satisfies the transverse hard Lefschetz property, and if it is also a Riemannian foliation,
then there exists a formal Frobenius manifold structure on its equivariant basic cohomology.
\subsection{dGBV algebra in transversely symplectic geometry}
We first give a quick review of \emph{differential Gerstenhaber-Batalin-Vilkovisky (dGBV) algebra}.
Suppose $(\mathscr{A},\wedge)$ is a supercommutative graded algebra with identity over a field $k$, and that
there is a $k$-linear operator $\delta:\mathscr{A}^{*}\rightarrow\mathscr{A}^{*-1}$.
Define the bracket $[\bullet]$ by setting
\[[a\bullet b]=(-1)^{|a|}\biggl(\delta(a\wedge b)-(\delta a)\wedge b-(-1)^{|a|}a\wedge(\delta b)\biggr),\]
where $a$ and $b$ are homogeneous elements and $|a|$ is the degree of $a\in\mathscr{A}$.
We say that $(\mathscr{A},\wedge,\delta)$ forms a \emph{Gerstenhaber-Batalin-Vilkovisky (GBV) algebra} with odd bracket $[\bullet]$,
if it satisfies:
\begin{itemize}
  \item [(1)] $\delta$ is a differential, i.e., $\delta^{2}=0$;
  \item [(2)] for any homogeneous elements $a$, $b$ and $c$ we have
  \begin{equation}\label{key-identity}[a\bullet(b\wedge c)]=[a\bullet b]\wedge c+(-1)^{(|a|+1)|b|}b\wedge[a\bullet c].\end{equation}
\end{itemize}
\begin{defn}\label{def7.2}
A GBV-algebra $(\mathscr{A},\wedge,\delta)$ is called a \emph{dGBV-algebra},
if there exists a differential operator $d:\mathscr{A}^{*}\rightarrow\mathscr{A}^{*+1}$ such that
\begin{itemize}
  \item [(1)] $d$ is a derivation with respect to the product $\wedge$, i.e., $d(a\wedge b)=da\wedge b+(-1)^{|a|}a\wedge db$
               for any homogeneous elements $a$ and $b$;
  \item [(2)] $d\delta+\delta d=0$.
\end{itemize}
\end{defn}

An \emph{integral} on a dGBV algebra $\mathscr{A}$ is a $k$-linear functional
\begin{equation}\label{integral}
\int:\mathscr{A}\rightarrow k
\end{equation}
such that for all $a,b\in\mathscr{A}$, the following equations hold
$$
\int(da)\wedge b=(-1)^{\mid a\mid+1}\int a\wedge db,
$$
$$
\int(\delta a)\wedge b=(-1)^{\mid a\mid}\int a\wedge\delta b.
$$
Moreover, an integral $\int$ induces a bi-linear pairing on $H(\mathscr{A},d)$ as follows.
\[ (\cdot, \cdot): H(\mathscr{A},d)\times H(\mathscr{A},d)\rightarrow k, \,\,\,([a],[b])=\int\, a\wedge b.\]
In particular, if the above bi-linear pairing is non-degenerate, then we say that the integral is \emph{nice}.

The following theorem enables us to use a dGBV algebra as an input to produce a formal Frobenius manifold (cf. \cite{BK98},\cite{Man99}).
\begin{thm}\label{thm7.3}
Let $(\mathscr{A},\wedge,\delta,d,[\bullet])$ be a dGBV algebra satisfying the following conditions:
\begin{itemize}
  \item [(1)] the dimension of $H(\mathscr{A},d)$ is finite;
  \item [(2)] there exists a nice integral on $\mathscr{A}$;
  \item [(3)] the inclusions $(\ker\,\delta,d)\hookrightarrow (\mathscr{A},d)$ and
              $(\ker\,d,\delta)\hookrightarrow (\mathscr{A},\delta)$ are quasi-isomorphisms.
\end{itemize}
Then there is a canonical construction of a formal Frobenius manifold structure on $H(\mathscr{A},d)$.
\end{thm}

As an initial step, we first prove that the equivariant basic Cartan complex of a transversely symplectic manifold carries the structure of a dGBV algebra.

\begin{prop}\label{dGBV-algebra}
Suppose that there is a transverse action of a compact connected Lie group $G$ on a transversely symplectic manifold $(M,\mathcal{F},\omega)$.
Let $\delta$ be the differential on equivariant basic differential forms as introduced in Section \ref{eq-formality-ddelta-lemma},
and let $\wedge$ denote the wedge product.
Then the quadruple $(\Omega_{G,\emph{bas}},\wedge, \delta, d_G)$ is a dGBV algebra.
\end{prop}

\begin{proof}
The only thing that requires a proof is  that (\ref{key-identity}) holds on equivariant basic differential forms.
To this end, it suffices to show that (\ref{key-identity}) holds for ordinary basic differential forms $a, b, c$ on a foliated coordinate neighborhood.
So without loss of generality, we may assume that $ b= f_0 df_1\wedge \cdots \wedge df_k$,
and that for each $\,0\leq i\leq k$, $f_i$ is a basic functions such that $df_i=\iota(X_i)\omega $ for some foliate vector field $X_i$.
However, it is easy to see that if $b_1,\cdots, b_s$ are basic forms such that for each $1\leq i\leq s$, (\ref{key-identity}) holds for $b=b_i$ and arbitrarily given basic forms $a$ and $c$, then (\ref{key-identity}) holds for $b=b_1\wedge\cdots \wedge b_s$ and arbitrarily given basic forms  $a$ and $c$.
Therefore it is enough to show that (\ref{key-identity}) is true in the following two cases.

\textbf{Case 1)} Assume that $b=f$ is a basic function such that $df=\iota(X)\omega$ for some foliate vector $X$.
Applying b) in Lemma \ref{Leibniz-rule}, we have
\[
\begin{split}
[a\bullet fc]&=(-1)^{\vert a\vert}\left( \delta(a\wedge fc)-\delta (a)\wedge fc-(-1)^{\vert a\vert}a\wedge \delta(fc)\right)
\\&=(-1)^{\vert a\vert}\left( f\delta(a\wedge c)-(\iota(X) a)\wedge c-\delta (a)\wedge fc-(-1)^{\vert a\vert} a\wedge f \delta c\right)
\\&= f[a\bullet c]-(-1)^{\vert a\vert} (\iota(X)a)\wedge c
\\&=f[a\bullet c]+(-1)^{\vert a\vert} (\delta(fa)-f\delta a)\wedge c
\\&=f[a\bullet c]+[a\bullet f]\wedge c.
\end{split}
\]

\textbf{Case 2)} Assume that $b=df$ for a basic function $f$ such that $df=\iota(X)\omega$ for some foliate vector $X$.
On the one hand, due to the identity c) in Lemma \ref{Leibniz-rule}, we get
\begin{equation}\label{step1}
\begin{split}
[a\bullet (df\wedge c)]&=(-1)^{\vert a\vert}\left( \delta( a\wedge df\wedge c)-\delta a\wedge df\wedge c-(-1)^{\vert a\vert}a\wedge \delta(df\wedge c)\right)\\&=\mathcal{L}(X)(a\wedge c)-df\wedge \delta (a\wedge c)-(-1)^{\vert a\vert}\delta a \wedge df\wedge c+ a\wedge df \wedge \delta c-a\wedge \mathcal{L}(X)c
\\&=(\mathcal{L}(X)a) \wedge c- df\wedge \delta (a\wedge c)+ df\wedge \delta a\wedge c+ a\wedge df\wedge \delta c
\\&=(\mathcal{L}(X)a)\wedge c - df\wedge\left (\delta (a\wedge c)-\delta a\wedge c-(-1)^{\vert a\vert}a\wedge \delta c\right)
\\&=(\mathcal{L}(X)a)\wedge c+(-1)^{\vert a\vert+1} df\wedge [a\bullet c].
\end{split}
\end{equation}
On the other hand, applying c) in Lemma \ref{Leibniz-rule} again, we have
 \begin{equation}\label{step2}
 \begin{split}
 [a\bullet df]
   &=(-1)^{\vert a\vert}\left( \delta(a\wedge df)-\delta a\wedge df -(-1)^{\vert a\vert} a\wedge\delta df\right)\\
   &=\delta( df\wedge a)-(-1)^{\vert a\vert}\delta a \wedge df +a\wedge d\delta f\\
   &=-df\wedge \delta a+\mathcal{L}(X)a+df\wedge \delta a\\
   &=\mathcal{L}(X)a
 \end{split}
 \end{equation}
It follows immediately from (\ref{step1}) and (\ref{step2}) that (\ref{key-identity}) holds in this case.
\end{proof}

\subsection{Formal Frobenius manifolds from dGBV-algrbras}
To show that there is a nice integral on the $dGBV$-algebra $(\Omega_{G,\textmd{bas}},\wedge, \delta, d_G)$,
we need the transverse integration theory developed on the space of basic forms on a taut Riemannian foliation (cf.  \cite[Chapter 7]{T97}, \cite{Ser85}).
Here we follow the method used in \cite{T97}, as it may be easier to describe for a general audience.

Recall that a foliation $\mathcal{F}$ on a smooth manifold $M$ is said to be \emph{Riemannian}, if there exists a Riemannian metric $g$ on $M$, called a \emph{bundle-like} metric for the foliation $\mathcal{F}$,
such that for any two foliate vector fields $Y$ and $Z$ on an open subset $U\subset M$ which are perpendicular to the leaves,
the function $g(Y,Z)$ is basic on $U$ (cf. \cite{R59a}).
From now on, we assume that $M$ is a closed oriented connected smooth manifold,
that $(M,\mathcal{F},\omega)$ is a transversely symplectic foliation of dimension $l$ and co-dimension $2n$ which satisfies the transverse hard Lefschetz property,
and that there is a Hamiltonian action \[G\times M\rightarrow M,\, (h,x)\mapsto L_h(x)\] of a compact connected Lie group $G$ on $M$.
In addition, we also assume that $\mathcal{F}$ is a Riemannian foliation with a bundlelike metric $g$.

Let $P$ be the integrable subbundle of $TM$ associated to the foliation $\mathcal{F}$ on $M$.
Observe that under our assumption $\mathcal{F}$ is transversely oriented.
It follows that $\mathcal{F}$ is also tangentially oriented.
That is to say that $P$ is an oriented vector bundle.
Fix an orientation on $P$, and define the \emph{characteristic form} $\chi_{\mathcal{F}}$ for the triple $(M, g,\mathcal{F})$ as follows (cf. \cite[Chapter 4]{T97}).
\begin{equation}\label{char-form}
\chi_{\mathcal{F}}(Y_{1},\cdots,Y_{l})=\det\bigl(g(Y_{i},E_{j})\bigr),
\end{equation}
where $Y_{1},\cdots,Y_{l}\in T_{x}M$, and $(E_{1},\cdots,E_{l})$ is an oriented orthonormal frame of $P_x$.
Clearly, when the orientation on $P$ is fixed, the definition of $\chi_{\mathcal{F}}$ depends only on the choice of a bundlelike metric.
However, by the transverse hard Lefschetz property, $H^{2n}(M,\mathcal{F})\cong H^0(M,\mathcal{F})\cong \mathbb{R}$, which implies that the Riemannian foliation $(M,\mathcal{F})$ is taut (cf. \cite[Theorem 1.4.6]{PAW09}).
Thus as explained in \cite[Chapter 7]{T97}  and \cite[Formula 4.26]{T97}, we can choose a bundlelike metric $g$ such that the corresponding characteristic form
$\chi_{\mathcal{F}}$ satisfies
\begin{equation}\label{taut-consequence}
\iota(X_1)\cdots \iota(X_l) d\chi_{\mathcal{F}}=0\,\,\,\forall\, X_1,\cdots, X_l\in C^{\infty}(P).
\end{equation}

Since the action of $G$ preserves the foliation $\mathcal{F}$, it is easy to check that $\forall\, h\in G$, the characteristic form with respect to the  pullback metric $L_h^* g$ is $L_h^*\chi_{\mathcal{F}}$.
A straightforward check shows that $L_h^*\chi_{\mathcal{F}}$  also satisfies (\ref{taut-consequence}).
So averaging the bundlelike metric $g$ over the compact Lie group $G$ if necessary,
we may assume that the characteristic form $\chi_{\mathcal{F}}$ with respect to the bundlelike metric $g$ is not only $G$-invariant,
but also satisfies (\ref{taut-consequence}).
In particular, $\chi_{\mathcal{F}}$ can be regarded as an equivariant differential form.
Using the usual equivariant integration (cf. \cite{GS99}), we define a $S(\mathfrak{g^{*}})^{G}$-linear operator as follows.
\begin{equation}\label{integral-operator}
  \int: \Omega_{G,\textmd{bas}}\rightarrow  S(\mathfrak{g^{*}})^{G},\quad
  \alpha \mapsto \int_{M}\alpha\wedge {\chi}_{\mathcal{F}}.
\end{equation}

\begin{lem}\label{integral-identity-on-dGBV}
 $\forall\,\alpha\in\Omega^{s}_{G,\emph{bas}}$, $\forall\,\beta\in\Omega^{t}_{G,\emph{bas}}$, we have that
\begin{itemize}
\item [a)]
\begin{equation}\label{basic-integral-dG}
\int(d_{G}\alpha)\wedge \beta=(-1)^{s+1}\int \alpha\wedge d_{G}\beta,
\end{equation}
\item [b)]
\begin{equation}\label{basic-integral-delta}
\int(\delta \alpha)\wedge \beta=(-1)^{s}\int \alpha\wedge\delta \beta.
\end{equation}
\end{itemize}
\end{lem}

\begin{proof}
We first prove a preliminary result that for any two ordinary basic differential forms $\alpha \in \Omega^s(M,\mathcal{F})$ and
$\beta\in \Omega^t(M,\mathcal{F})$ the following identity holds.
\begin{equation}\label{coho-pairing-preliminary}
\int_M(d\alpha)\wedge \beta\wedge \chi_{\mathcal{F}}=(-1)^{s+1}\int_M \alpha\wedge d\beta\wedge \chi_{\mathcal{F}},
\end{equation}
By the Leibniz rule,
\[
d(\alpha\wedge\beta\wedge {\chi}_{\mathcal{F}})=
d\alpha\wedge\beta\wedge \chi_{\mathcal{F}}
+(-1)^{s}\alpha\wedge(d\beta)\wedge \chi_{\mathcal{F}}
+(-1)^{s+t}\alpha\wedge\beta\wedge d \chi_{\mathcal{F}}.
\]
Since
\[
  \int_{M}d(\alpha\wedge\beta\wedge{\chi}_{\mathcal{F}})=0,
\]
to prove (\ref{coho-pairing-preliminary}) it suffices to show that
\begin{equation}\label{equ7.3}
\int_M\,\alpha\wedge\beta\wedge d\chi_{\mathcal{F}}=0.
\end{equation}
Observe that $\chi_{\mathcal{F}}$ is of degree $l$, we may assume that $s+t=2n-1$,
for otherwise (\ref{equ7.3}) holds for degree reasons.
Next recall that by our choice of the bundle-like metric, the characteristic form $\chi_{\mathcal{F}}$ has the property that for any vector fields
$X_{1},\cdot\cdot\cdot,X_{l}$ tangent to the leaves of $\mathcal{F}$,
$
\iota(X_{1})\cdots\iota(X_{l})d\chi_{\mathcal{F}}=0.
$
Since $\alpha$ and $\beta$ are basic, this would imply that  $\alpha\wedge\beta\wedge d\chi_{\mathcal{F}}=0$,
from which  (\ref{coho-pairing-preliminary}) follows as an immediate consequence.

Since $d$ does not act on the polynomial part of an equivariant basic form, (\ref{coho-pairing-preliminary}) also holds for equivariant basic forms.  On the other hand, for each $\alpha\in \Omega^s_{G}(M,\mathcal{F})$ and $\beta\in\Omega^t_{G}(M,\mathcal{F})$,
a simple degree counting shows that
\begin{equation}\label{degree-counting}
\int_M\partial \alpha\wedge\beta\wedge d\chi_{\mathcal{F}}=\int_M\alpha\wedge \partial \beta\wedge d\chi_{\mathcal{F}}=0.
\end{equation}
Combing (\ref{coho-pairing-preliminary}) and (\ref{degree-counting}) we get that (\ref{basic-integral-dG}) holds.

To prove the assertion b), it suffices to show that for any ordinary basic forms $\alpha\in \Omega^s(M,\mathcal{F})$
and $\beta \in \Omega^t(M,\mathcal{F})$,
$$
\int_M(\delta\alpha)\wedge\beta\wedge \chi_{\mathcal{F}}=(-1)^{s}\int\alpha\wedge(\delta\beta)\wedge \chi_{\mathcal{F}}.
$$
Without loss of generality, we may assume that $s+t=2n+1$.
Using (\ref{star-adjoint-identity}) and (\ref{coho-pairing-preliminary}),
we have
\begin{equation*}
\begin{split}
\int_M(\delta\alpha)\wedge\beta\wedge \chi_{\mathcal{F}}&=(-1)^{s+1}\int_M (\star d\star \alpha) \wedge\beta\wedge \chi_{\mathcal{F}}\\&=(-1)^{s+1}\int_M(d\star\alpha)\wedge \star \beta \wedge \chi_{\mathcal{F}}
\\&=\int_M (\star \alpha)\wedge d\star \beta\wedge \chi_{\mathcal{F}}
\\&=(-1)^{s}\int_M \alpha \wedge \delta\beta \wedge \chi_{\mathcal{F}}.
\end{split}
\end{equation*}
This completes the proof.
\end{proof}

Note that $S(\mathfrak{g}^*)^G$ is an integral domain.
Let  $\mathbb{F}=\{\frac{f}{g}\,\vert\, f, g\in S(\mathfrak{g}^*)^G\}$ be the fractional field of $S(\mathfrak{g}^*)^G$.
Define
\[
\widetilde{\Omega}_{G,\textmd{bas}}=\Omega_{G,\textmd{bas}}\otimes_{S(\mathfrak{g}^*)^G}\mathbb{F}.
\]
Extend $d_G, \wedge$ and $\delta$ to  $\widetilde{\Omega}_{G,\textmd{bas}}$, and define
\begin{equation}\label{normalized-eq-basic-coho}
\widetilde{H}_{G}(M,\mathcal{F})= H(\widetilde{\Omega}_{G,\textmd{bas}}, d_G).
\end{equation}
As an direct consequence of Theorem \ref{eq-formality}, we have
\[
\widetilde{H}_{G}(M,\mathcal{F})= H_{G}(M,\mathcal{F})\otimes_{S(\mathfrak{g}^*)^G}\mathbb{F}.
\]

Applying Proposition \ref{dGBV-algebra}, we see that $(\widetilde{\Omega}_{G,\textmd{bas}}, \delta, \wedge, d_G)$ is a dGBV-algebra over $\mathbb{F}$.
Moreover, the operator defined in (\ref{integral-operator}) naturally extends to a $\mathbb{F}$-linear operator
\begin{equation}\label{integral-on-dGBV}
\int: \widetilde{\Omega}_{G,\textmd{bas}}\rightarrow \mathbb{F}.
\end{equation}
 Clearly, Lemma \ref{integral-identity-on-dGBV} implies that the above operator (\ref{integral-on-dGBV}) defines an integral on the dGBV algebra $(\widetilde{\Omega}_{G,\textmd{bas}},\wedge,\delta,d_G)$.
To show that this integral is also nice, we need the following result on the basic Poincar\'{e} duality.
\begin{thm}[{\cite[Corollary 7.58]{T97}}]\label{basic-Poincare-duality}
Let $\mathcal{F}$ be a taut and transversally oriented Riemannian foliation on a closed oriented manifold $M$.
The the pairing
$$
\alpha\otimes\beta\mapsto\int_{M}\alpha\wedge\beta\wedge\chi_{\mathcal{F}}
$$
induces a non-degenerate pairing
$$
H^{r}(M,\mathcal{F})\times H^{q-r}(M,\mathcal{F})\rightarrow\mathbb{R}
$$
on finite-dimensional vector spaces,
where $q=\emph{codim}\,\mathcal{F}$.
\end{thm}

\begin{lem} \label{non-degeneracy}
The integral operator defined in (\ref{integral-on-dGBV}) is nice, i.e.,
it induces a $\mathbb{F}$-bi-linear non-degenerate pairing
\[
\widetilde{H}_{G}^*(M,\mathcal{F})\times \widetilde{H}^*_{G}(M,\mathcal{F})\rightarrow \mathbb{F}.
\]
\end{lem}

\begin{proof}
Let $[\alpha]$ be an arbitrary class in $ H_{G}(M,\mathcal{F})$ such that
\[
\int_M \alpha \wedge \beta \wedge \chi_{\mathcal{F}}=0, \,\,\,\textmd{for each}\, [\beta]\in H_{G}(M,\mathcal{F}).
\]
To prove Lemma \ref{non-degeneracy}, it suffices to show that $[\alpha]$ has to vanish.

Let $\{f_1,\cdots ,f_k\}$ be a basis of the real vector space $(S\mathfrak{g}^*)^G$.
By Theorem \ref{eq-formality}, there exist finitely many  cohomology classes $[\gamma_i]$'s in $H(M,\mathcal{F})$ such that
\[
 [\alpha]= \displaystyle \sum_i f_i\otimes s([\gamma_i]).
\]
Here $s: H(M,\mathcal{F})\rightarrow H_{G}(M,\mathcal{F})$ is the canonical section introduced in (\ref{canonical-section}).
Let $k_i$ be the degree of the basic form $\gamma_i$.
After a reshuffling of the index, we may assume that $k_1\geq k_2\geq \cdots$.
Then for any $[\zeta] \in H^{2n-k_1}(M,\mathcal{F})$, we have that
\[
\displaystyle \sum_if_i\otimes \left( \int_M s([\gamma_i]) \wedge s([\zeta])\wedge\chi_{\mathcal{F}}\right)=0,\]
which implies that
\[ \int_{M}s([\gamma_{1}])\wedge s([\zeta])\wedge\chi_{\mathcal{F}}=0.
\]

It then follows from a simple counting of degrees that $\int_M\gamma_1\wedge \zeta \wedge \chi_{\mathcal{F}}=0$.
Since $[\zeta]\in H^{2n-k_1}(M,\mathcal{F})$ is arbitrarily chosen, by Theorem \ref{basic-Poincare-duality} we have that $[\gamma_1]=0$.
Thus $s([\gamma_1])=0$.
Repeating this argument, we see that $[\gamma_i]=0$ for all $i$.
It follows that $[\alpha]$ must be zero.
\end{proof}

We are ready to state the main result of this section.
\begin{thm}\label{construction-Frobenius-manifolds-eq-case}
Assume that $(\mathcal{F},\omega)$ is a transversely symplectic foliation on a closed oriented smooth manifold $M$ that satisfies the transverse hard Lefschetz property,
and that a compact connected Lie group $G$ acts on $(M,\mathcal{F},\omega)$ in a Hamiltonian fashion.
If $\mathcal{F}$ is also a Riemannian foliation, then there is a canonical formal Frobenius manifold structure on the equivariant basic cohomology $\widetilde{H}_{G}(M,\mathcal{F})$ as defined in (\ref{normalized-eq-basic-coho}).
\end{thm}
\begin{proof}
It remains to show that the following maps induced by the inclusions
\begin{equation}\label{eq7.1}
\rho:H(\ker\,\delta,d_{G})\rightarrow H(\Omega_{G,\textmd{bas}},d_{G})
\end{equation}
\begin{equation}\label{eq7.2}
\mu:H(\ker\,d_{G},\delta)\rightarrow H(\Omega_{G,\textmd{bas}},\delta)
\end{equation}
are isomorphisms.
The fact that the map (\ref{eq7.1}) is an isomorphism is a direct consequence of Theorem \ref{eq-formality}.
Let $\alpha\in\ker\,d_{G}$ be a $\delta$-closed form which represents a class $[\alpha]$ in $H(\ker\,d_{G},\delta)$.
Suppose that $[\alpha]$ is trivial in $H(\Omega_{G,\textmd{bas}},\delta)$,
then there exists a $\beta\in\Omega_{G,\textmd{bas}}$ such that $\alpha=\delta\beta$.
By Theorem \ref{dG-delta-lemma}, we have $\alpha=d_{G}\delta\gamma$ for some $\gamma\in\Omega_{G,\textmd{bas}}$. This shows that $\alpha$ represents a trivial class in $H(\ker\,d_{G},\delta)$, and that the map (\ref{eq7.2}) is injective.

To see that (\ref{eq7.2}) is surjective, suppose that  $\alpha\in \Omega_{G,\textmd{bas}}$ such that $\delta \alpha=0$,
i.e., $[\alpha]$ is a class in $ H(\Omega_{G,\textmd{bas}},\delta)$.
Let $\gamma=d_{G}\alpha$. Then $\gamma$ is both $d_G$-exact and $\delta$-closed.
By Theorem \ref{dG-delta-lemma}, there exists a $\beta\in\Omega_{G,\textmd{bas}}$ such that $\gamma=d_{G}\delta\beta$.
Set $\tilde{\alpha}=\alpha-\delta\beta$.
Then we have that $\tilde{\alpha}\in\ker\,d_{G}$, and that $[\tilde{\alpha}]=[\alpha]$ in $H(\Omega_{G,\textmd{bas}},\delta)$.
This proves that (\ref{eq7.2}) is surjective.
By Theorem \ref{thm7.3} there exists a formal Frobenius manifold structure on $\widetilde{H}_{G}(M,\mathcal{F})$.
\end{proof}

When $G$ is a trivial group consisting of one single element, we have the following result.
\begin{cor} \label{non-eq-case}
Assume that $(M,\mathcal{F},\omega)$ is a transversely symplectic manifold that satisfies the transverse hard Lefschetz property.
If $\mathcal{F}$ is also a Riemannian foliation,
then there is a canonical formal Frobenius manifold structure on the  basic cohomology $H(M,\mathcal{F})$.
\end{cor}

\begin{rem}
When the foliation $\mathcal{F}$ is zero dimensional, we recover from Corollary \ref{non-eq-case} the Merkulov's construction \cite{Mer98} of a Frobenius manifold structure on the de Rham cohomology of a symplectic manifold with the hard Lefschetz property.
When the foliation $\mathcal{F}$ is zero dimensional, and when $M$ is a closed K\"ahler manifold,
we recover from Theorem \ref{construction-Frobenius-manifolds-eq-case} the construction by Cao-Zhou \cite{CZ99},
which produces a Frobenius manifold structure on the equivariant cohomology of a Hamiltonian action of a compact connected Lie group on a K\"ahler manifold.
Moreover, we are able to remove the assumption  in \cite{CZ99} that the action is holomorphic.
\end{rem}

\section{Examples of Frobenius manifolds from transversely symplectic foliations}

In this section we  present some examples of transversely symplectic foliations which
give rise to new examples of dGBV-algebra whose cohomology admits a formal Frobenius manifold structure.
We begin with an useful observation on when a locally free action of a compact Lie group gives rise to a $G$-invariant Riemannian foliation.

\begin{lem} \label{Rie-foliation-condition}
Consider the locally free action of a compact connected Lie group $G$ on a manifold $M$.
Let $\mathfrak{g}=\text{Lie}\,(G)$, let $\mathfrak{h}$ be an ideal of $\mathfrak{g}$,
and let $\mathcal{F}$ be the foliation generated by the infinitesimal action of $\mathfrak{h}$ on $M$.
Then $\mathcal{F}$ is a $G$-invariant Riemannian foliation.
\end{lem}

\begin{proof}
It is clear from our assumption that the foliation $\mathcal{F}$ is $G$-invariant.
Now suppose that $g$ is an $G$-invariant Riemannian metric.
We will show that $g$ must be a bundle-like metric.
Let $Y$ and $Z$ be two foliate vector fields which are perpendicular to the leaves, and let
$\xi_M$ be the fundamental vector field generated by the infinitesimal action of $\xi\in\mathfrak{h}$.
Then we have that
\[  \mathcal{L}(\xi_M) \left(g(Y, Z)\right)=(\mathcal{L}(\xi_M)g)(Y, Z)+g([\xi_M,X], Y)+g(X, [\xi_M,Y]).\]

Note that $\mathcal{L}(\xi_M)g=0$ because $g$ is $G$-invariant.
Moreover, since $X$ is a foliate vector field, $[\xi_M, X]$ must be tangent to the leaves.
Thus $g([\xi_M,X], Y)=0$ as $Y$ is perpendicular to the leaves.
For the same reason we have that $g(X, [\xi_M,Y])=0$.
It follows that
$\mathcal{L}(\xi_M) \left(g(Y, Z)\right)=0$.
Since $\xi\in \mathfrak{h}$ is arbitrarily chosen, $g(Y, Z)$ must be a basic function.
This completes the proof.
\end{proof}

Now we discuss examples of transversely symplectic foliations to which Theorem \ref{construction-Frobenius-manifolds-eq-case}
and Corollary \ref{non-eq-case} apply.

\begin{ex}[Co-oriented contact manifolds]\label{contact-mflds}
Let $M$ be a $2n+1$ dimensional co-oriented compact contact manifold with a contact one form $\eta$ and a Reeb vector $\xi$.
Then the Reeb characteristic foliation $\mathcal{F}_{\xi}$ induced by $\xi$ is transversely symplectic,
with a transversely symplectic form $d\eta$.
If there exists a contact metric $g$ such that $\xi$ is a Killing vector field, then $(M,\eta,g)$ is called a $K$-contact manifold.
It is well known that the Reeb characteristic foliation of a $K$-contact manifold $(M,\eta,g)$ is Riemannian.
By Corollary \ref{non-eq-case}, when $M$ satisfies the transverse hard Lefschetz property, its basic cohomology will carry the structure of a formal Frobenius manifold.
In particular, this is the case when $(M,\eta,g)$ is a Sasakian manifold (cf. \cite{BG08}).
It is also noteworthy that there exist examples of compact $K$-contact manifolds
which do not admit any Sasakian structures, and which satisfy the hard Lefschetz property as introduced in \cite{CNY13,CNMY15}.
By \cite[Theorem 4.4]{L13}, these non-Sasakian $K$-contact manifolds also satisfy the transverse hard Lefschetz property.
\end{ex}

\begin{ex}[Hamiltonian actions on contact manifolds]
Let $M$ be a $2n+1$ dimensional compact contact manifold with a contact one form $\eta$ and a Reeb vector field $\xi$,
and let $G$ be a compact connected Lie group with the Lie algebra $\mathfrak{g}$.
Suppose that $G$ acts on $M$ preserving the contact one form $\eta$.
Then the $\eta$-contact moment map $\Phi: M\rightarrow \mathfrak{g}^*$,
given by
\[
\langle\Phi,X\rangle=\eta(X_M),\,\,\, \forall\, X\in \mathfrak{g},
\]
also defines a moment map for the transverse $G$-action on the transversely symplectic foliation $(M,\mathcal{F}_{\xi},d\eta)$.
Here $\langle\cdot, \cdot\rangle$ is the dual pairing between $\mathfrak{g}$ and $\mathfrak{g}^*$,
and $X_M$ is the fundamental vector field generated by $X$.

Recall that the action of $G$ is said to be of \emph{Reeb type}, if the Reeb vector $\xi$ is generated by the infinitesimal action of an element in $\mathfrak{g}$ (cf. \cite[Definition 8.4.28]{BG08}).
It is clear from Lemma \ref{Rie-foliation-condition} that when the action of $G$ is of Reeb type, the Reeb characteristic foliation $\mathcal{F}_{\xi}$ is Riemannian.
If in addition, $(M,\eta,g)$ is a Sasakian manifold, then $\mathcal{F}_{\xi}$ satisfies the transverse hard Lefschetz property.
In particular, these observations apply to the case when $(M,\eta,g)$ is a compact toric contact manifold of Reeb type.
Therefore by Theorem \ref{construction-Frobenius-manifolds-eq-case} there is a formal Frobenius manifold structure on the equivariant basic cohomology of toric contact manifolds of Reeb type.
\end{ex}

\begin{ex}[Co-symplectic manifolds \cite{Li08}]\label{co-sym}
Let $(M,\eta, \omega)$ be a $(2n+1)$-dimensional compact co-symplectic manifold.
By definition,  $\eta$ is a closed one form, and $\omega$ a closed two form $\omega$, such that $\eta\wedge \omega^n$ is a volume form.
Then the Reeb characteristic foliation $\mathcal{F}_{\xi}$ induced by the Reeb vector field $\xi$ (defined by the equations $\iota(\xi)\eta=1$ and $\iota(\xi)\omega=0$) is transversely symplectic with the transversely symplectic form $\omega$.

We claim that for any $1\leq k\leq n$, the basic form $\omega^k$ represents a non-trivial basic cohomology class in $H^{2k}(M,\mathcal{F})$.
Assume to the contrary that $[\omega^k]=0\in H^{2k}(M,\mathcal{F})$ for some $1\leq k\leq n$.
Then there exists a basic $(2n-1)$-form $\beta$ such that $\omega^n=d\beta$.
Since $d\eta=0$, we have
\[
\int_M \,\eta\wedge \omega^n=\int_M\,\eta\wedge d\beta=\int_M\, -d(\eta\wedge \beta)=0,
\]
which contradicts the fact that $\eta\wedge\omega^n$ is a volume form.
This proves our claim.

The co-symplectic manifold $M$ is called a co-K\"ahler manifold, if one can associate to $(M, \eta, \omega)$ an almost contact structure
$(\phi,\xi, \eta, g)$, where $\phi$ is an $(1,1)$-tensor, and $g$ a Riemannian metric, such that $\phi$ is parallel with respect to the Levi-Civita connection of $g$.
It is straightforward to check that if $M$ is co-K\"ahler, then the Reeb characteristic foliation $\mathcal{F}_{\xi}$ is transversely K\"ahler.
Due to the claim established in the previous paragraph, it is indeed a taut transversely K\"ahler foliation, and therefore satisfies the transverse hard Lefschetz property.
By Corollary \ref{non-eq-case}, the basic cohomology of $M$ has a structure of a formal Frobenius manifold.
\end{ex}

\begin{ex}[Symplectic orbifolds]\label{sym-orbifolds}
Let $(X,\sigma)$ be an effective symplectic orbifold of dimension $2n$.
Then the total space of the orthogonal frame orbi-bundle $\pi:\emph{Fr}(X)\rightarrow X$ is a smooth manifold on which the structure group $\textmd{O}(2n)$ acts locally free.
The form $\omega:=\pi^{*}\sigma$ is a closed 2-form on $\emph{Fr}(X)$ whose kernel gives rise to a transversely symplectic foliation $\mathcal{F}$. It follows easily from Lemma \ref{Rie-foliation-condition} that $\mathcal{F}$ is also Riemannian.
When $X$ is a K\"ahler orbifold, it was shown in \cite{WZ09} that $\emph{Fr}(X)$ satisfies the transverse hard Lefschetz property.
Since in this case, the basic differential complex of $(\emph{Fr}(X),\mathcal{F})$ is isomorphic to the de Rham differential complex on $X$,
Corollary \ref{non-eq-case} implies that there is a formal Frobenius manifold structure on the de Rham cohomology of $X$.

Now suppose that a compact connected Lie group $G$ acts on $(X,\sigma)$ in a Hamiltonian fashion with a moment map
$\Phi:X\rightarrow \mathfrak{g}^*$, where $\mathfrak{g}=\emph{Lie}(G)$.
By averaging, we may assume that there is a $G$-invariant Riemannian metric $g$ that is compatible with $\sigma$.
Then the $G$-action maps an orthogonal frame to another orthogonal frame;
and therefore, lifts to a Hamiltonian $G$-action on $(\emph{Fr}(X),\mathcal{F},\omega)$.
Analogous to the discussion in the previous paragraph,
when $X$ is K\"ahler orbifold, Theorem \ref{construction-Frobenius-manifolds-eq-case} implies that there is a formal Frobenius manifold structure on the equivariant de Rham cohomology of $X$.
\end{ex}

\begin{ex}[Symplectic quasi-folds \cite{Pra01}]\label{sym-quasi-folds}
Assume that $(X,\sigma)$ is a symplectic manifold on which the torus $T$ acts in a Hamiltonian fashion.
We denote the moment map by
$
\phi:X\rightarrow\mathfrak{t}^{*}.
$
Let $N\subset T$ be a non-closed subgroup with Lie algebra $\mathfrak{n}$ and let $a$ be a regular value of the corresponding moment map
$
\varphi:X\rightarrow\mathfrak{n}^{*}.
$
Consider the submanifold $M=\varphi^{-1}(a)\subset X$.
The $N$-action on $M$ yields a transversely symplectic foliation $\mathcal{F}$
with $\omega:=i^{*}\sigma$ being the transversely symplectic form,
where $i$ is the inclusion map of $M$ in $X$.
In this case, the leaf space $M/\mathcal{F}$ is a symplectic quasi-fold in the sense of Prato \cite{Pra01},
at least when $N$ is a connected subgroup of $T$.
It is straightforward to check that the induced $T$-action on $(M,\mathcal{F},\omega)$ is Hamiltonian.

It follows from Lemma \ref{Rie-foliation-condition} that $\mathcal{F}$ is also a Riemannian foliation.
Moreover, using an argument similar to the one given in Example \ref{co-sym}, it can be shown
that  $\mathcal{F}$ is a taut Riemannian folation.
The leaf space of $\mathcal{F}$ is called a toric quasi-fold when $\text{dim}\, (T/N)$ is half of the dimension of the leaf space.
It is shown by Ishida \cite[Theorem 5.7]{Is17} that when this is the case, $\mathcal{F}$ is a transversely K\"ahler foliation.
Therefore there exist formal Frobenius manifold structures on the basic cohomology and equivariant basic cohomology of toric quasi-folds.
\end{ex}


\end{document}